\documentclass[11pt]{amsart}

\usepackage{amsfonts,amssymb,amsthm,palatino,euler}
\usepackage{stmaryrd}
\usepackage[left=2.0cm,right=2.0cm,top=3.0cm,bottom=3.0cm]{geometry}
\usepackage{cite}

\usepackage{thm-restate}
\usepackage[all,cmtip]{xy}
\usepackage{epsfig}
\usepackage{capt-of} 
\usepackage{multicol}
\usepackage{xfrac} 

\usepackage[colorlinks=true]{hyperref} 

\title{Overgroups of exterior powers of an elementary group.\\Normalizers}

\author{Roman~Lubkov}
\address[Roman~Lubkov]{Saint Petersburg University, 7/9 Universitetskaya nab., St. Petersburg, 199034 Russia}
\email{r.lubkov@spbu.ru, RomanLubkov@yandex.ru}

\dedicatory{In memory of\\
	Nikolai Aleksandrovich Vavilov,\\
	our teacher, a brilliant mathematician, and a generous colleague}

\author{Ilia~Nekrasov}
\address[Ilia~Nekrasov]{Department of Mathematics, University of Michigan, Ann Arbor, MI}
\email{inekras@umich.edu, geometr.nekrasov@gmail.com}

\thanks{The first author is supported by ``Native towns'', a social investment program of PJSC ``Gazprom Neft'' and by the Ministry of Science and Higher Education of the Russian Federation (agreement no. 075–15–2022–287).}


\date{}
\let\opn\operatorname
\DeclareMathOperator{\E}{E}
\DeclareMathOperator{\M}{M}

\DeclareMathOperator{\SO}{SO}
\DeclareMathOperator{\GO}{GO}
\DeclareMathOperator{\SL}{SL}
\DeclareMathOperator{\GL}{GL}
\DeclareMathOperator{\Sp}{Sp}
\DeclareMathOperator{\GSp}{GSp}
\DeclareMathOperator{\sign}{sgn}
\DeclareMathOperator{\Plu}{Plu}
\DeclareMathOperator{\Fix}{Fix}
\DeclareMathOperator{\Lie}{Lie}
\DeclareMathOperator{\Tran}{Tran}

\renewcommand{\trianglelefteq}{\trianglelefteqslant}
\renewcommand{\leq}{\leqslant}
\renewcommand{\geq}{\geqslant}

\newcommand{\bw}[1]{\mathord{\raisebox{2pt}
		{\hbox{$\scriptstyle{\bigwedge^{\!#1}}$}}}}
\newcommand\blank{\mathord{\hbox to 1.5ex{\hrulefill}}\,}

\theoremstyle{plain}
\newtheorem{theorem}{Theorem}
\newtheorem{prop}[theorem]{Proposition}
\newtheorem{lemma}[theorem]{Lemma}
\newtheorem{corollary}[theorem]{Corollary}

\theoremstyle{remark}
\newtheorem{remark}[theorem]{Remark}

\begin{document}
	
	\begin{abstract}
         We establish two characterizations of an algebraic group scheme $\bw{m}\GL_n$ over $\mathbb{Z}$. Geometrically, the scheme $\bw{m}\GL_n$ is a stabilizer of an explicitly given invariant form or, generally, an invariant ideal of forms. Algebraically, $\bw{m}\GL_n$ is isomorphic (as a scheme over $\mathbb{Z}$) to a normalizer of the elementary subgroup functor $\bw{m}\E_n$ and a normalizer of the subscheme $\bw{m}\SL_n$.

         Our immediate goal is to apply both descriptions in the ``sandwich classification'' of overgroups of the elementary subgroup. Additionally, the results can be seen as a solution of the linear preserver problem for algebraic group schemes over $\mathbb{Z}$, providing a more functorial description that goes beyond geometry of the classical case over fields.
	\end{abstract}
    \vspace*{-1.2cm} 
	\maketitle

	\section*{Introduction}\label{sec:intro}
	
	The present work is a sequel of~\cite{LubNek-OverI} where we have started the description of overgroups of exterior powers of an elementary group. In this paper, we carry out the second key step of the description: an explicit calculation of the normalizer of elementary groups in the corresponding general linear group.
	
	In the case when $n$ is a multiple of $m$, we construct an $R$--linear form $f \colon V\times\dots\times V \longrightarrow R$ in $k$ variables, where $V=V(\varpi_m)$ and $R$ is an arbitrary commutative ring. We prove that $\bw{m}\SL_n$ coincides with the algebraic group $G_f$ of linear transformations preserving this form $f$:
	$$G_f(R):=\Big\{g\in\GL_{\binom{n}{m}}(R)\mid f(gx^1,\ldots,gx^k)=f(x^1,\ldots,x^k)\Big\}.$$
	We deliver analogous description for $\bw{m}\GL_n$ in terms of the form $f$. Namely, this group scheme is equal to the stabilizer $	\overline{G}_f$ of the ideal generated by the form $f$:
	\begin{align*}
		\overline{G}_f(R)&:=\Big\{g\in\GL_{\binom{n}{m}}(R)\mid g \text{ preserves the ideal }\langle f\rangle\Big\}.
	\end{align*}
	
	\begin{restatable}{theorem}{exGLandForms}
		\label{thm:exGLandForms}
		If $\sfrac{n}{m}$ is an integer greater than 2, then there are isomorphisms $\bw{m}\SL_n\cong G_f$, $\bw{m}\GL_n\cong\overline{G}_f$ of affine group schemes over $\mathbb{Z}$.
	\end{restatable}
	
	The theorem follows the traditional description of a Chevalley group as a stabilizer of a form and the corresponding extended Chevalley group as the stabilizer up to a scalar multiplier, see \cite{ChevalleyClassificationDG5658}.
	
	In the case when $n$ is not divisible by $m$, we construct an ideal $F$, a direct generalization of $\langle f\rangle$, such that $\bw{m}\GL_n$ coincides with a stabilizer of this ideal:
	\begin{align*}
		\overline{G}_F(R)&:=\Big\{g\in\GL_{\binom{n}{m}}(R)\mid g \text{ preserves the ideal }F \Big\}.
	\end{align*}
	
	\begin{restatable}{theorem}{exGLandFormsIdeal}
		\label{thm:exGLandFormsIdeal}
		Using prior notation, $\bw{m}\GL_n$ and $\overline{G}_F$ are isomorphic as affine group scheme over $\mathbb{Z}$.
	\end{restatable}
    Analogous description for the general case of $n$ and $m$ can be found in \cite{VavPere}, as we discuss in Section~\ref{sec:StabPlu}. Indeed, the group scheme $\bw{m}\GL_n$ is a stabilizer of the Pl\"ucker ideal $\Plu$ generated by Pl\"ucker \textit{quadratic} forms. However, our description goes further then just taking a subideal of $\Plu$: ideal $F$ from Theorem~\ref{thm:exGLandFormsIdeal} is a proper subideal of the radical $\sqrt{\Plu}$ with some nice properties.
 
    In the theory of linear preserver problem and, more generally, in geometric invariant theory there exists a classic geometric interpretation of a normalizer $N_{\GL(V)}(G)$ of a group $G$ acting irreducibly on a vector space $V$: $N_{\GL(V)}(G)$ is equal to $\mathrm{Stab}_{\GL(V)}(\mathcal{O})$, where $\mathcal{O}$ is a closed $G$-orbit in $\mathbb{P}(V)$ (we invite reader to consult~\cite[Theorem 3.2.]{BerGariLar2014} and references there). Theorems~\ref{thm:exGLandForms} and~\ref{thm:exGLandFormsIdeal} can be seen as an example of a \textit{scheme-theoretic} incarnation of the statement. Authors hope to pursue this direction for wider class of groups in a future publication.
	
	Now let $C,D$ be two subgroups of an abstract group $G$. Recall that the \textit{transporter} of $C$ to $D$ is the set:
	$$ 
	\Tran_G(C,D)=\{g\in G\mid C^g \le D\}. 
	$$
 
    We need a scheme-theoretic analogue~\cite[Section V.6.]{MilneAGS}: \textit{scheme-theoretic transporter} of $X$ to $Y$ inside an algebraic group $G$ is the functor $\Tran_{G}(X,Y)$ such that
    \[
    \Tran_{G}(X,Y)(R) = \{g\in G(R)\mid z^g\in Y(\tilde{R}) \text{ for all $R$-algebras }\tilde{R}\text{ and } z\in X(\tilde{R})\}.
    \]
    The scheme-theoretic normalizer $N_G(X)$ is defined as a scheme-theoretic transporter $\Tran_{G}(X,X)$.
 
    We denote the elementary subgroup of $\GL_n(R)$ by $\E_n(R)$ and the corresponding $m$-th exterior power of the elementary group by $\bw{m}\E_n(R)$. The following is our second result.
 
	\begin{theorem}\label{thm:NNTran}
        If $n \geq 4$ and $\sfrac{n}{m}$ is an integer greater than 2, then there are isomorphisms of the affine algebraic group schemes over $\mathbb{Z}$:
        $$N\bigl(\bw{m}\E_n\bigr) \cong N\bigl(\bw{m}\SL_{n}\bigr) \cong \Tran\bigl(\bw{m}\E_n, \bw{m}\SL_{n}\bigr) \cong \Tran\bigl(\bw{m}\E_n, \bw{m}\GL_{n}\bigr) \cong \bw{m}\GL_{n},$$
	where all scheme-theoretic normalizers and transporters are taken inside $\GL_{\binom{n}{m}}$.
	\end{theorem}
 
    According to the results of \cite{LubStepSub} and forthcoming \cite{LubStepSubE}, we can replace the normalizers and transporters with their group-theoretic analogues for some classes of rings $R$. For example, $\Tran\bigl(\bw{m}\E_n(R), \bw{m}\SL_{n}(R)\bigr)$ coincides with $\Tran\bigl(\bw{m}\E_n, \bw{m}\SL_{n}\bigr)(R)$ for algebras $R$ over infinite fields, see \cite[Proposition 4.3]{LubStepSub}. In other words, the classic version of Theorem~\ref{thm:NNTran} with abstract transporters holds as well over these rings, see~\cite[Theorem~3]{VP-Ep,VP-EOodd,VavLuzgE7}, \cite[Theorem~2]{VavLuzgE6}, \cite[Theorem~4]{AnaVavSinI} for analogues in other cases.

	The paper is organized as follows. In~\S\ref{sec:definition} we present the basic notation. 
	We recall the well-known description using the Pl\"ucker polynomials  in~\S\ref{sec:StabPlu}, we construct an invariant form for $\bw{m}\GL_n$ for the case $\sfrac{n}{m}\in\mathbb N$ in~\S\ref{sec:StabI}, and, in~\S\ref{sec:StabII}, we generalize the latter description to an invariant system of forms for any $n,m$. \S\ref{sec:DiffExterPwrs} gives a geometric description of the quotient $\bw{m}\GL_n(R)$ by $\bw{m}\bigr(\GL_n(R)\bigr)$. Finally, in~\S\ref{sec:transgood} we discuss different notions of normalizers and transporters and prove Theorem~\ref{thm:NNTran}.
	
	\section{Exterior powers of elementary groups}\label{sec:definition}
	
	In this section, we introduce exterior powers of an elementary group and define the related concepts.
	
	We denote the set $\{1,2,\dots, n\}$ by $[n]$. \textit{If there is no confusion, we denote the binomial coefficient $\binom{n}{m}$ by $N$}. Elements of $\bw{m}[n]$, the $m$-th exterior power of the set $[n]$, are ordered subsets $I\subseteq [n]$ of cardinality $m$ without repeating entries:
	$$\bw{m}[n] = \{(i_{1},\dots,i_{m})\mid 1\leq i_1<i_2<\dots<i_m\leq n \}.$$
	
	Let $R$ be a commutative ring and let $R^n$ be the right free $R$-module with the standard basis $\{e_1,\dots,e_n\}$. $\bw{m}R^n$ is a free module of rank $N=\binom{n}{m}$ with the basis $e_{i_1}\wedge\dots\wedge e_{i_m}$ with $(i_1,\dots,i_m) \in \bw{m}[n]$. The products $e_{i_1}\wedge\dots\wedge e_{i_m}$ are defined for an arbitrary set $\{i_1,\dots,i_m\}$ via $e_{\sigma(i_{1})}\wedge\dots\wedge e_{\sigma(i_{m})} = \sign(\sigma)\, e_{i_1} \wedge \dots \wedge e_{i_{m}}$ for $\sigma \in S_m$ a permutation of $[m]$. We can assume that  $n \geq 2m$ due to the isomorphism $\bw{m}V^{*} \cong (\bw{\dim(V)-m}V)^{*}$ for an arbitrary free $R$-module $V$.
	
	For every $m\leq n$, we have Cauchy--Binet homomorphism $\bw{m}\colon \GL_n(R) \longrightarrow \GL_N(R)$ defined via the diagonal action:
	$$\bw{m}(g)(e_{i_1}\wedge\dots\wedge e_{i_m}):=(ge_{i_1})\wedge\dots\wedge (ge_{i_m}) \text{ for }e_{i_1},\dots,e_{i_m}\in R^n.$$
	Thus $\bw{m}$ is a representation of the group $\GL_n(R)$. It is called the $m$-th \textit{vector representation} or the $m$-th \textit{fundamental representation}. The image group $\bw{m}\bigr(\GL_n(R)\bigr)$ is called the $m$-th exterior power of the general linear group. 
	
	By $a_{i,j}$ we denote an entry of a matrix $a\in\GL_n(R)$ at the position $(i,j)$, where $1\leq i,j\leq n$. Further, $e$ denotes the identity matrix and $e_{i,j}$ denotes the standard matrix unit, i.e. the matrix that has $1$ at the position $(i,j)$ and zeros elsewhere. For entries of the inverse matrix we use the standard notation $a_{i,j}':=(a^{-1})_{i,j}$. The \textit{$[$absolute$]$ elementary group} $\E_n(R)$ is a subgroup of $\GL_n(R)$ generated by all elementary transvections $t_{i,j}(\xi)=e+\xi e_{i,j}$, where $1\leq i\neq j\leq n$, $\xi\in R$. The set $\E^l(n,R)$ is a subset of $\E_n(R)$ consisting of products of at most $l$ elementary transvections. The exterior power of the elementary group $\bw{m}\E_n(R)$ is defined as the $\bw{m}$--image of the elementary group $\E_n(R)$.
	
	In the sequel, we use weight diagrams to illustrate internal combinatorics of equations. We refer the reader to~\cite{atlas} where the authors describe all the rules to construct weight diagrams. The exterior power of the elementary group $\bw{m}\E_n(R)$ corresponds to the representation of the Chevalley group of type $\Phi=A_{n-1}$ with the highest weight $\varpi_m$.
	
	In the majority of existing constructions, $\bw{m}\GL_n(R)$ arises together with an action on the Weyl module $V(\varpi_m) = R^N$. 
	We denote the weight set of the module $V(\varpi_m)$ by $\Lambda(\varpi_m)$. Then $\Lambda(\varpi_m) = \bw{m}[n]$. 
	
	Fix an admissible base $v^{\lambda},\lambda\in\Lambda$ of the module $V=V(\varpi_m)$. We regard a vector $a\in V$, $a=\sum v^{\lambda}a_{\lambda}$, as a column of coordinates $a=(a_{\lambda}),\lambda\in\Lambda$. 
	
	In fig.~\ref{fig:L2E5} and fig.~\ref{fig:L3E7} we reproduce the weight diagrams of the groups $\bw{2}\E_5(R)$ and $\bw{3}\E_7(R)$, which correspond to representations $(A_4,\varpi_2)$ and $(A_6,\varpi_3)$, respectively. We follow the convention of naturally ascending numbering of weights. On the diagrams, the highest weight is the leftmost one. Recall that in a weight diagram two weights are joined by an edge if their difference is a fundamental root.

	\begin{figure}[t]
		\centering 
			\includegraphics[scale=0.7]{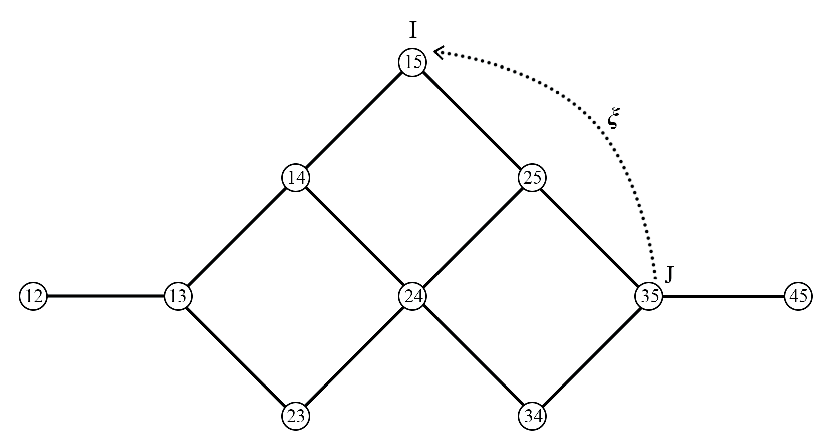}
			\caption{Weight diagram $(A_4,\varpi_2)$ and action of $t_{I,J}(\xi)$}
			\label{fig:L2E5}
	\end{figure}
	

	
	\begin{figure}[t]
		\centering
			\includegraphics[scale=1.3]{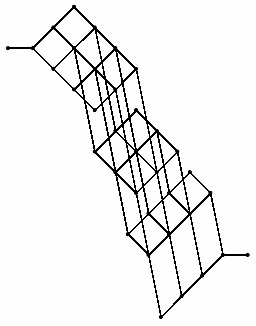}
			\caption{Weight diagram $(A_6,\varpi_3)$}
			\label{fig:L3E7}
	\end{figure}
	

	The algebraic group scheme $\bw{m}\GL_{n}$ is, by definition, the \textit{categorical} image of the group scheme $\GL_{n}$ under Cauchy--Binet homomorphism. The group $\bw{m}\GL_{n}(R)$ is defined as $R$--points of the functor $\bw{m}\GL_{n} = \bw{m}\GL_{n}(\blank)$. The [abstract] groups  $\bw{m} \bigl(\GL_{n}(R)\bigr)$ and $\bw{m}\GL_{n}(R)$ are different for a general ring $R$. We have a canonical inclusion $\bw{m} \bigl(\GL_{n}(R)\bigr) \leqslant \bw{m}\GL_{n}(R)$; the quotient set is computed in \S\ref{sec:DiffExterPwrs}.
	
	Abstractly, elements of $\bw{m}\GL_n(R)$ are images of matrices under $\bw{m}$ with entries belonging to some extension of $R$. In other words, arbitrary element $\widetilde{g}\in\bw{m}\GL_n(R)$ has the form $\widetilde{g}=\bw{m}g$ where $g \in \mathrm{GL}_n(S)$ for some extension ring $S$ of $R$.
	
	Below we show that the group $\bw{m}\SL_n(R)$ is the standard Chevalley group $G(\Phi,R)$, $\bw{m}\GL_n(R)$ is the extended Chevalley group $\overline{G}(\Phi,R)$, and $\bw{m}\E_n(R)$ coincides with the [absolute] elementary subgroup of $G(\Phi,R)$. 
	
	Recall that $\bw{m}\E_n(R)$ is normal not only in the image of the general linear group but in the bigger group $\bw{m}\GL_n(R)$. This fact follows from~\cite[Theorem~1]{PetrovStavrovaIsotropic}.
	\begin{theorem}\label{thm:normalityPS}
		Let $R$ be a commutative ring, $n\geq 3$, then
		$\bw{m}\E_n(R)\trianglelefteq\bw{m}\GL_n(R)$.
	\end{theorem}
	
	We recall the explicit form of the exterior power of an elementary transvection~\cite{LubNek-OverI} which we utilize later.
	
	\begin{prop}\label{prop:ImageOfTransvForm}
		Let $t_{i,j}(\xi)$ be an elementary transvection in $\E_n(R)$, $n\geq 3$. Then
		\begin{equation}
			\bw{m}t_{i,j}(\xi)=\prod\limits_{L\in\bw{m-1}\bigl([n]\smallsetminus \{i,j\}\bigr)} t_{L\cup i,L\cup j}\bigl(\sign(L, i)\sign(L, j)\xi\bigr)
			\label{eq:m}
		\end{equation}
		for any $1\leq i\neq j\leq n$.
	\end{prop}
	
	Similarly, one can get an explicit form of the torus elements $h_{\varpi_m}(\xi)$ of the group $\bw{m}\GL_{n}(R)$.
	\begin{prop}\label{prop:ImageOfDiag}
		Let $d_i(\xi)=e+(\xi-1)e_{i,i}$ be a torus generator, $1\leq i\leq n$. Then the exterior power of $d_i(\xi)$ equals the diagonal matrix with diagonal entries 1 everywhere except in $\scriptstyle{\binom{n-1}{m-1}}$ positions:
		\begin{equation}
			\bw{m}\bigl(d_i(\xi)\bigr)_{I,I}=
			\begin{cases}
				\xi,& \text{ if } i\in I,\\
				1,& \text{ otherwise}.
			\end{cases}
			\label{eq:diagm}
		\end{equation}
	\end{prop}
	
	As an example, consider $\bw{3}t_{1,3}(\xi)=t_{124,234}(-\xi)t_{125,235}(-\xi)t_{145,345}(\xi)\in\bw{3}\E_5(R)$ and $\bw{4}d_2(\xi)=\opn{diag}(\xi,\xi,\xi,1,\xi)\in\bw{4}\E_5(R)$. It follows from the propositions that $\bw{m}t_{i,j}(\xi)$ belongs to $\E^{\binom{n-2}{m-1}}(N,R)$. In other words, the residue\footnote{The \textit{residue} $\opn{res}(g)$ of a transformation $g$ is, by definition, the rank of $g-e$.} of an exterior transvection $\opn{res}\bigr(\bw{m}t_{i,j}(\xi)\bigr)$ equals the binomial coefficient $\binom{n-2}{m-1}$. 
	
	Let $I,J$ be two elements of $\bw{m}[n]$. We define a \textit{distance} between $I$ and $J$ as the cardinality of the intersection $I\cap J$:
	$$d(I,J)=|I\cap J|.$$
	This combinatorial characteristic plays an analogous role of the distance function $d(\lambda,\mu)$ for roots $\lambda$ and $\mu$ on the weight diagram of a root system. 
	
	\section{Stabilizer of the Pl\"ucker ideal}\label{sec:StabPlu}
	
	First we recall the well-known description of polyvector representations of the general linear group. In~\cite{VavPere} the authors proved that $\bw{m}\GL_n(R)$ coincides with the stabilizer of the Pl\"ucker ideal. 
	
	Pl{\"u}cker polynomials are homogeneous quadratic polynomials $f_{I,J}\in \mathbb{Z}\bigr[x_H, H\in \bw{m}[n]\bigr]$ of Grassmann coordinates $x_H$. In general, Pl{\"u}cker polynomials can be represented in the form:
	$$f_{I,J}=\sum\limits_{j\in J\backslash I}\pm x_{I\cup \{j\}}x_{J\backslash \{j\}},$$
	where $I\in\bw{m-1}[n]$ and $J\in\bw{m+1}[n]$. To clarify the sign of the factors, we extend the definition of the Grassmann coordinates as follows. If there are coinciding elements in the set $\{i_1,\dots ,i_m\}$, then $x_{i_1\dots i_m}=0$; otherwise $x_{i_1\dots i_m}=\sign(i_1,\dots,i_m)x_{\{i_1\dots i_m\}}$. Thus the Pl{\"u}cker polynomials have the form:
	$$f_{I,J}=\sum\limits_{h=1}^{m+1}(-1)^hx_{i_1\dots i_{m-1}j_h}x_{j_1\dots \hat{j}_h\dots j_{m+1}}.$$
	A \textit{Pl{\"u}cker ideal }$\Plu:=\Plu_{n,m}\trianglelefteq R\bigl[x_{I}: I \in \bw{m}[n]\bigr]$ is generated by all Pl{\"u}cker relations $f_{I,J}$ with $I\in\bw{m-1}[n]$ and $J\in\bw{m+1}[n]$.
	
	\begin{lemma}\label{lem:VPLEsavesPlu}
		Let $R$ be an arbitrary commutative ring. The group $\bw{m}\E_n(R)$ preserves the Pl{\"u}cker ideal $\Plu$.
	\end{lemma}
	
	Following notation of the paper~\cite{VavPere}, we put $G_{nm}(R):=\Fix_R(\Plu)$ for any commutative ring $R$, where $\Fix_R(\Plu)$ is the set of $R$-linear transformations preserving the ideal $\Plu$:
	$$G_{nm}(R):=\{g\in\GL_N(R)\mid f(gx)\in\Plu\text{ for all }f\in\Plu\}.$$
	\begin{lemma}\label{lem:VPPluscheme}
		For any $n,m$ the functor $R\mapsto \Fix_R(\Plu)$ is an affine group scheme defined over $\mathbb{Z}$.
	\end{lemma}
	
	Next results are classical known, see~\cite{ChevalleyClassificationDG5658} and~\cite[Theorem 4]{WaterhousePGL}. Note that representation $\bw{m}$ is minuscule. Therefore it is irreducible and tensor indecomposable.
	\begin{lemma}\label{lem:VPKer}
		Let $K$ be an algebraically closed field. For any $n,m$ with $1\leq m\leq n-1$, the kernel of $\bw{m}$ for $\GL_n(K)$ and $\SL_n(K)$ equals $\mu_m$ and $\mu_d$ where $d=\gcd(n,m)$, respectively.
	\end{lemma}
	\begin{lemma}\label{lem:VPIrrTIandNorm}
		As a subgroup of $\GL_N(K)$, the algebraic group $\bw{m}\bigl(\GL_n(K)\bigr)$ is irreducible and tensor indecomposable. Moreover, except the case $n=2m\geq 4$, the group $\bw{m}\bigl(\GL_n(K)\bigr)$ coincides with its normalizer. In the exceptional case, the group has index $2$ in its normalizer. 
		
		The analogous result holds for $\bw{m}\bigl(\SL_n(K)\bigr)$ as a subgroup of $\SL_N(K)$. 
	\end{lemma}
	
	Using the classification of maximal subgroups in classical groups by Gary Seitz~\cite[Table 1]{SeitzMaxSub} (see also the survey~\cite{burness_testerman_CorrSeitz} with corrections), it is easy to prove that $\bw{m}\SL_n(K)$ is maximal for an algebraically closed field $K$. The following statement is Lemma~7 of the paper~\cite{VavPere}.
	\begin{lemma}\label{lem:VPMax}
		Let $K$ be an algebraically closed field. For any $n,m, 1\leq m\leq n-1$ the groups $\bw{m}\GL_n(K)$ and $\bw{m}\SL_n(K)$ are maximal among connected closed subgroups in one of the following groups:
		\begin{multicols}{2}
			\begin{center} $\bw{m}\GL_n(K):$ \end{center}
            \begin{tabbing}
			 \par\noindent $\bullet$ in $\GL_N(K)$,\quad\=if $n\neq 2m$;\\
			 \par\noindent $\bullet$ in $\GSp_N(K)$,\quad\>if $n=2m$ $\mathbin{\&}$ odd $m$;\\
			 \par\noindent $\bullet$ in $\GO^0_N(K)$,\quad\>if $n=2m$ $\mathbin{\&}$ even $m$.
            \end{tabbing}			
			\columnbreak
			\begin{center} $\bw{m}\SL_n(K):$ \end{center}
            \begin{tabbing}
			 \par\noindent $\bullet$ in $\SL_N(K)$,\quad\=if $n\neq 2m$;\\
			 \par\noindent $\bullet$ in $\Sp_N(K)$,\quad\>if $n=2m$ $\mathbin{\&}$ odd $m$;\\
			 \par\noindent $\bullet$ in $\SO_N(K)$,\quad\>if $n=2m$ $\mathbin{\&}$ even $m$.
            \end{tabbing}
		\end{multicols}

		Besides, in the exceptional cases these classical groups are unique proper connected overgroups of $\bw{m}\GL_n(K)$ and $\bw{m}\SL_n(K)$, respectively.
	\end{lemma}
	\begin{corollary}\label{cor:VPLGLeqStabPlu}
		Suppose $K$ is an algebraically closed field; then $\bw{m}\GL_n(K)=G^0_{nm}(K)$.
	\end{corollary}
	
	Finally, for the coincidence of the group schemes, we must prove that $G_{nm}$ is smooth or, what is essentially the same, to calculate the
	dimension of the Lie algebra $\Lie(G_{nm})$.
	\begin{lemma}\label{lem:VPdimLie}
		For any field $K$ the dimension of the Lie algebra $\Lie(G_{nm,K})$ does not exceed $n^2$.
	\end{lemma}
	Using Theorem~1.6.1 of~\cite{WaterhouseDet}, we get the following result.
	\begin{theorem}\label{thm:VPWaterhouse}
		For any $n,m, 1\leq m\leq n-1$ there is an isomorphism of affine groups schemes over $\mathbb{Z}$:
		$$
		G_{nm}\cong
		\begin{cases}
			\GL_n/\mu_m,& \text{ if } n\neq 2m,\\
			\GL_n/\mu_m\leftthreetimes\mathbb{Z}/2\mathbb{Z},& \text{ if } n=2m.
		\end{cases}
		$$
	\end{theorem}
	
	\section{Exterior powers as the stabilizer of invariant forms I}\label{sec:StabI}
	
	Next we present an alternative description of $\bw{m}\GL_n(R)$ as a stabilizer of a form. Analogous forms are well known for classical and exceptional groups in the standard representation over an arbitrary ring, see~\cite{VP-EOeven,VP-EOodd,VP-Ep,VavLuzgE6,VavLuzgE7}. Conveniently for the reader, a general approach was developed  by Skip Garibaldi and Robert Guralnick~\cite{GariGura2015,GarGur22}. We also refer to~\cite[Section~$4.4$]{BermHern} where the author constructed cubic invariant forms for $\bw{m}\SL_n$.
    
    The following theorem is classically known and can be found in~\cite[Chapter 2, Sections 5--7]{DieCarInvTheor} for characteristic 0 and can be deduced from~\cite{Dem73, Veld72} as all primes are \textit{almost very good} in type $A_n$ or, nicely summarized, \cite[Theorem 1 (4)]{MirRum99} for fields of positive characteristic.
	\begin{prop}\label{prop:FormsGood}
        Let $K$ be an algebraically closed field. Then $\bw{m}\GL_{n}(K)$ is the group of similarities of an invariant form only in the case $\sfrac{n}{m} \in \mathbb{N}$ and $\sfrac{n}{m}\geq 3$. Moreover, this form is unique in the space of $\sfrac{n}{m}$-tensors and it is equal to
		\begin{itemize}
			\item $q^m_{[n]}(x) = \sum \sign(I_{1}, \ldots, I_{\frac{n}{m}})\; x_{I_{1}}\ldots x_{I_{\frac{n}{m}}}$ for even $m;$
			\item $q^m_{[n]}(x) = \sum \sign(I_{1}, \ldots, I_{\frac{n}{m}})\; x_{I_{1}}\wedge \ldots \wedge x_{I_{\frac{n}{m}}}$ for odd $m$,
		\end{itemize}
		where the sums in the both cases range over all unordered partitions of the set $[n]$ into $m$-element subsets $I_{1}, \ldots, I_{\frac{n}{m}}$.
	\end{prop}
	
	Henceforth, we use the uniform notation $q(x)$ for these forms and we assume that $m$ is even (unless otherwise specified); the case of odd $m$ can be addressed analogously. 
	
	So in the case of an algebraically closed field $K$,  the abstract group $\bw{m}\GL_{n}(K)$ consists of matrices $g\in\GL_N(K)$ for which there is a multiplier function $\lambda = \lambda(g) \in K^*$ such that $q(gx)=\lambda(g)q(x)$ for all $x\in K^N$. The calculation of $\lambda$ on a generic diagonal matrix $d_i(\xi)\in\GL_n(K)$ shows that $\lambda(g)=\det(g)$. Since the coefficients of these forms equal $\pm 1$, the forms are defined over $\mathbb{Z}$. The same calculation confirms the answer over an arbitrary ring:
	$$q(\bw{m}g \cdot x) = \det(g) \cdot q(x) \text{ for  }g \in \GL_{n}(R).$$

	To get a direct analog of Proposition~\ref{prop:FormsGood} over arbitrary rings, we change our focus from forms of high degree to the corresponding multilinear forms. Concretely, let $k:=\frac{n}{m}\in\mathbb{N}$, then a [full] \textit{polarization} for the forms $q(x) = q_{[n]}^m(x)$ is a $k$-linear form $f_{[n]}^m$:
	$$f(x) = f^m_{[n]}(x^1,\dots,x^k)=\sum \sign(I_{1}, \dots, I_k)\; x^1_{I_{1}}\dots x^{k}_{I_k},$$
	where the sum ranges over all \textit{ordered} partitions of the set $[n]$ into $m$-element subsets. 
	\begin{prop}\label{prop:FormsInvariantUnderLE}
		Let $R$ be an arbitrary commutative ring and $\sfrac{n}{m} \in \mathbb{N}$. The form $f$ is invariant under the action of $\bw{m}\E_n(R)$ and it is multiplied by $\xi$ under the action of a weight element $\bw{m}d_i(\xi)$.
	\end{prop}
	\begin{proof}
		As we noted previously, the multiplier $\lambda(g)$ is equal to the determinant. Indeed, $\lambda(g)$ is a one-dimensional representation, i.e. is a homomorphism $\GL_n(R)\longrightarrow\GL_1(R)$. Moreover, $\lambda(g)$ is a polynomial map that equals the determinant of $g$ over $\mathbb{C}$. Thus $\lambda(g)=\det(g)$ for an arbitrary ring $R$. And then the statement is obvious. But below we prove the proposition by direct calculation.
		
		We show that $f(gx^1,\dots,gx^k)=\xi f(x^1,\dots,x^k)$, where $g=\bw{m}d_i(\xi)$. Since $I_1,\ldots,I_k$ is an ordered partition of $[n]$, the number $i$ belongs to the index of only one variable $x^l_{I_l}$ in every monomial $x^1_{I_{1}}\dots x^{k}_{I_k}$ of the form $f$. Thus every monomial of $f(gx^1,\dots,gx^k)$ has the form $\pm x^1_{I_1}\ldots x^{l-1}_{I_{l-1}}\xi x^l_{I_l}x^{l+1}_{I_{l+1}}\ldots x^k_{I_k}$.
		
		Now let $g=\bw{m}t_{i,j}(\xi)$. By~\eqref{eq:m} the matrix $g$ is equal to the product of transvections $t_{iL,jL}(\sign(i,L)\sign(j,L)\xi)$ with $L\in\bw{m-1}\bigl([n]\smallsetminus \{i,j\}\bigr)$. Therefore exactly $\scriptstyle{\binom{n-2}{m-1}}$ coordinates change in the vector $gx, x\in R^N$: $(gx)_{iL}=x_{iL}+\sign(i,L)\sign(j,L)\xi x_{jL}$. Then in the form $f(gx^1,\dots,gx^k)-f(x^1,\dots,x^k)$ all monomials have the form:
		$$\pm x^1_{I_1}\ldots x^{l-1}_{I_{l-1}}\bigl(\sign(i,L)\sign(j,L)\xi x^l_{jL}\bigr)x^{l+1}_{I_{l+1}}\ldots x^k_{I_k},$$
		where $I_l=iL$, $L\in\bw{m-1}\bigl([n]\smallsetminus \{i,j\}\bigr)$. Let $I_1,\ldots,I_k$ be a partition of $[n]$ where $I_l=iL_1$, $I_p=jL_2$, $L_1,L_2\in\bw{m-1}\bigl([n]\smallsetminus \{i,j\}\bigr)$. Then the indices $\tilde{I}_1,\ldots,\tilde{I}_k$, where $\tilde{I}_l=jL_1$, $\tilde{I}_p=iL_2$, form a partition of $[n]$ as well. Therefore the sum of the corresponding monomials equals 
		\begin{align*}
			&\sign(I_1,\ldots,I_k)x^1_{I_1}\ldots x^{p}_{jL_2}\ldots x^{l-1}_{I_{l-1}}\bigl(\sign(i,L_1)\sign(j,L_1)\xi x^l_{jL_1}\bigr)x^{l+1}_{I_{l+1}}\ldots x^k_{I_k}+
			\\
			&\sign(\tilde{I}_1,\ldots,\tilde{I}_k)x^1_{\tilde{I}_1}\ldots x^l_{jL_1}\ldots x^{l-1}_{\tilde{I}_{l-1}}\bigl(\sign(i,L_2)\sign(j,L_2)\xi x^p_{jL_2}\bigr)x^{l+1}_{\tilde{I}_{l+1}}\ldots x^k_{\tilde{I}_k}
		\end{align*}
		It remains to check that the corresponding signs are opposite:
		$$\sign(I_1,\ldots,I_k)\sign(i,L_1)\sign(j,L_1)=-\sign(\tilde{I}_1,\ldots,\tilde{I}_k)\sign(i,L_2)\sign(j,L_2).$$
		Multiplying this equality by $\sign(j,L_2)\sign(j,L_1)$, we obtain
		$$\sign(I_1,\ldots,I_k)\sign(i,L_1)\sign(j,L_2)=-\sign(\tilde{I}_1,\ldots,\tilde{I}_k)\sign(i,L_2)\sign(j,L_1).$$
		And this is equivalent to 
		$$\sign(I_1,\ldots,I_k)=-\sign(\tilde{I}_1,\ldots,\tilde{I}_k),$$
		where the indices $I_p,\tilde{I}_p$ and $I_l,\tilde{I}_l$ are unordered.
		
		If $m$ is even, then this equality is equivalent to $\sign(iL_1,jL_2)=-\sign(jL_1,iL_2)$.
		Since $iL_1,jL_2$ and $jL_1,iL_2$ differ by an odd number of transpositions, the signs are opposite. Similarly, $I_1,\ldots,I_k$ and $\tilde{I}_1,\ldots,\tilde{I}_k$ differ by an odd number of transpositions for odd $m$.
	\end{proof}

    We denote the ring of all polynomials in (families of) variables $x^1 = \{x^1_{I}\}_{I \in \bw{m}[n]}, \dots,\\ x^k = \{x^k_{I}\}_{I \in \bw{m}[n]}$ with $R$-coefficients by $R[x^1, \dots, x^k]$. We consider a $\mathbb{Z}^k$-grading on this ring given by sums of degrees in each of the families $x^1, \dots, x^k$, e.g. the form $f = f^{m}_{[n]}(x^1,\dots , x^k)$ has grading $(1,\dots,1)$ as exactly one of variables from each families appears in each monomial of $f$. The submodule of all forms with grading $(1,\dots, 1)$ we denote by $R[x^1, \dots, x^k]_{(1,\dots, 1)}$.
    
    Applying the calculations similar to the previous proof, we get the uniqueness result for $\bw{m}\E_n(R)$--semi-invariant forms.
    
    \begin{prop}\label{prop:FormsGoodRing}	
    Let $R$ be an arbitrary ring and suppose $\sfrac{n}{m} \in \mathbb{N}$. Then every $\bw{m}\E_n(R)$--semi-invariant form in the space of multilinear forms $R[x^1, \dots, x^k]_{(1,\dots, 1)}$ is a multiple of  $f = f^{m}_{[n]}(x^1,\dots , x^k)$.
    \end{prop}
    \begin{proof}
    Consider arbitrary $F(x^1, \dots, x^k) = \sum a_{I_1, \dots, I_k}x^1_{I_1}\dots x^k_{I_k} \in R[x^1, \dots, x^k]_{(1,\dots, 1)}$.

    We first prove that for each nonzero $a_{I_1, \dots, I_k}$ the coefficients $I_1, \dots, I_k$ form a partition of $[n]$. Assume that there exists $j \in [n]$ such that $j \not\in I_1 \cup \dots \cup I_k$ for some tuple $(I_1, \dots, I_k)$ with $a_{I_1, \dots, I_k} \neq 0$. Choose an arbitrary $i$ appearing in at least one $I_i$; without loss of generality, $I_1 = i L_1$. Action by $\bw{m}t_{ij}(\zeta)$ on the monomial $a_{I_1, \dots, I_k}x^1_{I_1}\dots x^k_{I_k}$ contains the monomial $\pm \zeta \cdot a_{I_1, \dots, I_k} x^1_{j L_1}\dots x^k_{I_k}$. This monomial appears only for $a_{I_1, \dots, I_k}x^1_{I_1}\dots x^k_{I_k}$ due to the conditions on $j$. We get a contradiction with the semi-invariancy of $F$, so each $j\in [n]$ appears in at least one $I_1, \dots, I_k$.

    As each $I_i$ has cardinality $m$, the cardinality of their union is at most $m \cdot k =n$. Therefore $a_{I_1, \dots, I_k} \neq 0$ implies that $\{I_i\}$ forms a (non-intersecting) partition of $[n]$.

    For $a_{I_1, \dots, I_k} \neq 0$, we take arbitrary $i\in I_1$ with $I_1 = i L_1$ and $j \in I_2$ with $I_2 = j L_2$. Then action of $\bw{m}t_{ij}(1)$ on $a_{I_1, \dots, I_k}x^1_{I_1}\dots x^k_{I_k}$ has the form $a_{I_1, \dots, I_k}x^1_{I_1}\dots x^k_{I_k} + \sign(j,L_1)\cdot a_{I_1, \dots, I_k}x^1_{j L_1}x^2_{I_2}\dots x^k_{I_k}$. The latter term does not appear in $F$ as $j L_1 \cap I_2 = j$, therefore we need to cancel it out to get the semi-invariancy. Then $\sign(j,L_1)\cdot a_{I_1, \dots, I_k}x^1_{j L_1}x^2_{I_2}\dots x^k_{I_k}$ is forced to be equal to $-\sign(j,L_2)\cdot a_{j L_1,i L_2, \dots, I_k}x^1_{j L_1}x^2_{j L_2}\dots x^k_{I_k}$ coming from the action on the monomial $a_{j L_1,i L_2 \dots, I_k}x^1_{j L_1}x^2_{i L_2}\dots x^k_{I_k}$. In other words, for every $i \neq j$ from the disjoint partition $i L_1 \sqcup j L_2 \sqcup \dots \sqcup I_k = [n]$ we get the equation: 
    $$\sign(j,L_1)\cdot a_{i L_1, j L_2, \dots, I_k}+\sign(j,L_2)\cdot a_{j L_1,i L_2, \dots, I_k} = 0.$$
    Thus the final step of Proposition~\ref{prop:FormsInvariantUnderLE} proof implies that every non-zero $a_{I_1, \dots, I_k}$ coincides with $\sign(I_1,\dots, I_k)\cdot a$ for some shared $a \in R$.
    \end{proof}
    
	Let us define a group $G_f(R)$ as the group of linear transformations preserving the form $f(x^1,\ldots,x^k)$:
	$$G_f(R):=\{g\in\GL_N(R)\mid f(gx^1,\ldots,gx^k)=f(x^1,\ldots,x^k)\}.$$
	It is an analogue of the Chevalley group for the exterior powers. We define an analogue of the extended Chevalley group:
	\begin{align*}
		\overline{G}_f(R):=\{g\in\GL_N(R)\mid\text{ there exists } &\lambda = \lambda(g)\in R^* \text{ such that }
		\\
		&\quad f(gx^1,\ldots,gx^k)=\lambda(g)f(x^1,\ldots,x^k)\}.
	\end{align*}
	The functors $R\mapsto\overline{G}_f(R)$ and $R\mapsto G_f(R)$ define affine group schemes over $\mathbb{Z}$. Combining Proposition~\ref{prop:FormsInvariantUnderLE} and the reasonings before it for all rings $R$, we have the morphism of group schemes:
	\begin{align*}
		\iota\colon \bw{m}\GL_{n} \longrightarrow \overline{G}_f \quad\text{ or, after Theorem~\ref{thm:VPWaterhouse},}\quad \iota\colon \GL_n/\mu_m \longrightarrow \overline{G}_f.
	\end{align*}
	
	Ideally, we can expect the group $\bw{m}\GL_{n}(R)$ to coincide with $\overline{G}_f(R)$  (and $\bw{m}\SL_{n}(R)$ to coincide with $G_f(R)$) in the case $\sfrac{n}{m} \in \mathbb{N}$. Theorem~\ref{thm:exGLandForms} is a precise form of the expectation:
 
	\begin{theorem}\label{thm:StabAndGenerealPower}
		If $\sfrac{n}{m}$ is an integer greater than 2, then the group $\bw{m}\GL_{n}(R)$ coincides with $\overline{G}_f(R)$, and $\bw{m}\SL_{n}(R)$ coincides with $G_f(R)$ for an arbitrary ring $R$.
	\end{theorem}
 
	\begin{remark} If $n = 2m$ and $2$ is not a zero-divisor, then $\overline{G}_f(R) = \GO_{N}(R)$ or $\GSp_{N}(R)$ depending on the parity of $m$. So in this case $\bw{m}\GL_{n}(R)$ is a subgroup of the orthogonal or the symplectic group, respectively. Moreover, if $(n,m) = (4,2)$, then $\GO_{6}(R)$ equals $\bw{2}\GL_{4}(R)$.

    In general case, stabilizer of a quadratic form and its polarization do not coincide. Therefore, we only have the inclusion $\GO_{N}(R) \leq \overline{G}_f(R)$ or $\GSp_{N}(R) \leq \overline{G}_f(R)$.
	\end{remark}
	
	The \textit{proof} of the theorem follows the classic Waterhouse Lemma~\cite[Theorem~$1.6.1$]{WaterhouseDet}. This result essentially reduces the verification of an isomorphism of affine group schemes to the isomorphism of their groups of points over algebraically closed fields and the dual numbers\footnote{Recall that the algebra $K[\delta]$ of dual numbers over a field is isomorphic as a $K$-module to $K\oplus K\delta$ with multiplication given by $\delta^2=0$.} over such fields. 
	
	We note that an alternative proof based on SGA~\cite{SGA}, Exp. VI$\_$b, Cor. 2.6 can be developed, but we do not pursue this direction here.
	
	\begin{lemma}\label{lem:Waterhouse}
		Let $G$ and $H$ be affine group schemes of finite type over $\mathbb{Z}$ where $G$ is flat, and let $\varphi\colon G\longrightarrow H$ be a morphism of group schemes.  Assume that the following conditions are satisfied for any algebraically closed field $K$:
		\begin{enumerate}
			\item $\dim(G_K)\geq\dim_K(\Lie(H_K))$,
			\item $\varphi$ induces monomorphisms of the groups of points $G(K)\longrightarrow H(K)$ and
			$G(K[\delta])\longrightarrow H(K[\delta])$,
			\item the normalizer $\varphi(G^0(K))$ in $H(K)$ is contained in $\varphi(G(K))$.
		\end{enumerate}
		Here $G^0$ denotes the connected component of the identity in $G$, $G_K$ denotes the extension of scalars of $G$, and $\Lie(H_K)$ denotes the Lie algebra of the scheme $H_K$. 
		
		Then $\varphi$ is an isomorphism of group schemes over $\mathbb{Z}$.
	\end{lemma}
	
	In the case under consideration, the preliminary assumptions on the schemes are satisfied. Indeed, the schemes are of finite type being subschemes of appropriate $\GL_n$. The flatness condition follows from smoothness of the Chevalley--Demazure scheme $G$. All groups $G^0_K$ are smooth connected schemes of the same dimension. Moreover, we showed in the previous section that the normalizer of $\bw{m}\GL_n(K)$ in $\GL_N(K)$ coincides with $\bw{m}\GL_n(K)$. Thus condition (3) holds true.
	
	As we mentioned above, Theorem~\ref{thm:VPWaterhouse} shows that instead of a morphism $\GL_n/\mu_m \longrightarrow \overline{G}_f$ we can consider the morphism (which we call $\iota$ as well) $ \bw{m}\GL_{n} \longrightarrow \overline{G}_f$. Then Proposition~\ref{prop:FormsInvariantUnderLE} shows that $\bw{m}\E_n(R)$ is a subgroup of $\overline{G}_f(R)$ (as abstract groups) for any ring $R$. A standard argument shows that $\bw{m}\E_n(R)$ is dense in $\bw{m}\GL_{n}(R)$ for any local ring $R$. Therefore $\iota$ is a monomorphism for any local ring $R$.  So condition (2) follows.
	
	For $R = K$, an algebraically closed field, we can prove an even stronger statement:
	
	\begin{prop}\label{prop:LGLeqStabgoodOverK}
		Suppose $K$ is an algebraically closed field and $n\neq 2m$; then
		$$\bw{m}\GL_n(K)=\overline{G}^0_f(K)\quad\textrm{ and }\quad \bw{m}\SL_n(K)=G_f(K).$$
	\end{prop}
	\begin{proof}
		The group $\bw{m}\GL_n(K)$ preserves the invariant form $f(x^1,\ldots,x^k)$ by Proposition~\ref{prop:FormsGood}, thus $\bw{m}\GL_n(K)\leq \overline{G}_f(K)$. Since $\bw{m}\GL_n(K)$ is connected, we have $\bw{m}\GL_n(K)\leq \overline{G}^0_f(K)$. Further, from Lemma~\ref{lem:VPMax} it follows that $\bw{m}\GL_n(K)$ is maximal among connected closed subgroups in $\GL_N(K)$. Since $\overline{G}_f(K)$ is a proper subgroup of $\GL_N(K)$, we obtain the reverse inclusion. For the group $\bw{m}\SL_n(K)$ the proof is similar. 
	\end{proof}
	
	To deal with condition (1), it only remains to evaluate the dimension of the Lie algebras $\overline{G}_f$ and $G_f$. First let us recall how the Lie algebra of the scheme $G_{nm}$ is defined, see~\S\ref{sec:StabPlu}. We follow the ideas of William Waterhouse~\cite[Lemmas $3.2$, $5.3$, and $6.3$]{WaterhouseDet}. 
	
	Let $K$ be an arbitrary field. Then Lie algebra  $\Lie((G_f)_K)$ of an affine group scheme $(G_f)_K$ is most naturally interpreted as the kernel of homomorphism $G_f(K[\delta])\longrightarrow G_f(K)$ sending $\delta$ to $0$, see~~\cite{Jantzen,Waterhouse1979IntroductionTA,HumphreyLinAlgGr,BorelLie}. Let $G$ be a subscheme of $\GL_n$. Then $\Lie(G_K)$ consists of all matrices $e+z\delta$, $z\in \M_n(K)$, satisfying the equations defining $G(K)$. Formally, the statement takes the following form when $G$ is the stabilizer of a system of polynomials:
	
	\begin{lemma}\label{lem:LieAlgEquas}
		Let $\varphi_1,\dots,\varphi_s\in K[x_1,\dots,x_t]$. Then a matrix $e+z\delta$ with $z\in \M_t(K)$ belongs to $\Lie(\Fix_K(\varphi_1,\dots,\varphi_s))$ if and only if
		$$ 
		\sum_{1\leq i,j\leq t}z_{ij}x_i \frac{\partial \varphi_h}{\partial x_j}=0
		$$
		for all $h=1,\dots,s$.
	\end{lemma}
	
	To illustrate the argument that will be utilized for Theorem~\ref{thm:dimLie}, we first provide an outline of the proof of Lemma~\ref{lem:VPdimLie}.
	
	\begin{proof}[Proof of Lemma~\ref{lem:VPdimLie}]
		We apply Lemma~\ref{lem:LieAlgEquas} to the case of the stabilizer of Pl\"ucker polynomials $f_{K,L}(x)$, where $K\in\bw{m-1}[n]$, $L\in\bw{m+1}[n]$. There are three types of equations on entries $z_{I,J}$, see~\cite[proof of Proposition~3]{VavPere}). 
		\begin{itemize}
			\item $d(I,J)\leq m-2$, so we are in the case $|I\cup J|\geq m+2$, and then $z_{I,J}=0$;
			\item $d(I,J)=d(M,H)=m-1$ and $I-J=H-M$, then $z_{I,J}=\pm z_{H,M}$;
			\item $d(I,J)=d(M,H)=m-1$ and $I-H=J-M$, then $z_{I,I}\pm z_{H,H}=\pm z_{J,J}\pm z_{M,M}$,
		\end{itemize}
		where indices $I\in\bw{m}[n]$ we conceive as roots of the corresponding representation, see the proof of Theorem~\ref{thm:dimLie} and the example next to this theorem for a detailed description of such approach.
		
		The first case does not contribute to dimension of the Lie algebra. Matrix entries $z_{I,J}$ from the second case give the contribution equal to $n(n-1)$. And the third case contributes no more than $n$ linearly independent variables. Summing up, we get the upper bound equal to $n^2$.
	\end{proof}
	
	We consider the schemes $G_f(K)$ and $\overline{G}_f(K)$. The Lie algebra $\Lie(G_f(K))$ consists of all matrices $g=e+y\delta$, $y\in\M_N(K)$, satisfying the condition $f(gx^1,\dots,gx^k)=f(x^1,\dots,x^k)$ for all $x^1,\dots,x^k\in K^N$. Similarly, $\Lie(\overline{G}_f(K))$ consists of all matrices $g=e+y\delta$ with $y\in\M_N(K)$ satisfying the condition
	$f(gx^1,\dots,gx^k)=\lambda(g)f(x^1,\dots,x^k)$ for all $x^1,\dots,x^k\in K^N$.
	\begin{theorem}\label{thm:dimLie}
		If $n\neq 2m$, then for any field $K$ the dimension of the Lie algebra $\Lie(\overline{G}_f(K))$ does not exceed $n^2$, whereas the dimension of the Lie algebra $\Lie(G_f(K))$ does not exceed $n^2-1$.
	\end{theorem}
	\begin{proof}
		First observe that the conditions on elements of the Lie algebra $\Lie(G_f(K))$ are obtained from the corresponding conditions for elements of $\Lie(\overline{G}_f(K))$ by substituting $\lambda(g)=1$. Let $g$ be a matrix satisfying the above conditions for all $x^1,\dots,x^k\in K^N$. Plugging in $g=e+y\delta$ and using that the form $f$ is $k$-linear, we get
		$$\delta\bigl(f(yx^1,x^2,\ldots,x^k)+\ldots+f(x^1,\ldots,x^{k-1},yx^k)\bigr)=(\lambda(g)-1)f(x^1,\ldots,x^k).$$
		Now we show that the entries of the matrix $y$ are subject to exactly the same linear dependencies, as in the case $G_{nm}$. By definition $f(e_{I_1},\dots,e_{I_k})=0$ for all indices $I_1,\dots,I_k\in\bw{m}[n]$, except the cases where $\{I_j\}$ is a partition of the set $[n]=I_1\sqcup\dots\sqcup I_k$.
		\begin{itemize}
			\item If $d(I,J)\leq m-2$, then $y_{I,J}=0$. Indeed, in this case then there is a set of pairwise disjoint indices $I_2,\dots,I_k\in\bw{m}\bigl([n]\smallsetminus I\bigr)$ such that $d(J,I_2)\geq 1,\,d(J,I_3)\geq 1$ and $d(J,I_4)=\dots=d(J,I_k)=0$. Put $x^1:=e_{J},x^l:=e_{I_l}, 2\leq l\leq k$. Then $f(x^1,yx^2,\dots,x^k)=\dots=f(x^1,x^2,\dots,yx^k)=0$. It follows that $f(yx^1,x^2,\dots,x^k)=\pm y_{I,J}=0$.
			\item If $d(I,J)=m-1$ and $I-J=H-M$, then $y_{I,J}=\pm y_{H,M}$. Here there is a set of pairwise disjoint indices
			$M,I_3,\dots,I_k\in\bw{m}\bigl([n]\smallsetminus I\bigr)$ such that $d(J,M)=1$ and $d(J,I_3)=\dots=d(J,I_k)=0$. Put $x^1:=e_{J},x^2:=e_{M},x^l:=e_{I_l}, 3\leq l\leq k$ and denote by $H$ the index $[n]\smallsetminus(J\cup I_2\cup\dots\cup I_k)$. Then $f(x^1,x^2,yx^3,\dots,x^k)=\dots=f(x^1,x^2,\dots,yx^k)=0$. It follows that $f(yx^1,x^2,\dots,x^k)+f(x^1,yx^2,x^3\dots,x^k)=0$. But $f(yx^1,x^2,\dots,x^k)=\sign(I,M,I_3,\dots,I_k)\cdot y_{I,J}$, and $f(x^1,yx^2,x^3,\dots,x^k)=\sign(J,H,I_3,\dots,I_k)\cdot y_{H,M}$.
			\item Finally, if $d(I,M)=m-1$ and $I-M=H-J$, then $y_{I,I}-y_{M,M}=y_{H,H}-y_{J,J}$. Indeed, there is a set of pairwise disjoint indices $I_3,\dots,I_k\in\bw{m}\bigl([n]\setminus(I\cup J)\bigr)$. In other words, $I,J,I_3,\dots,I_k$ is a partition of the set $[n]$. Put $x^1:=e_I,x^2:=e_J,x^l:=e_{I_l}$, where $3\leq l\leq k$. Then $$(\lambda(g)-1)=\delta(y_{I,I}+y_{J,J}+y_{I_3,I_3}+\dots+y_{I_k,I_k}).$$
			On the other hand, $H,M,I_3,\dots,I_k$ is a partition of $[n]$ too, where $I\cup J=H\cup M$. Substituting $x^1:=e_H,x^2:=e_M,x^l:=e_{I_l}$ for all $3\leq l\leq k$, we get
			$$(\lambda(g)-1)=\delta(y_{M,M}+y_{H,H}+y_{I_3,I_3}+\dots+y_{I_k,I_k}).$$
			Combining the obtained equalities, we see $y_{I,I}+y_{J,J}=y_{M,M}+y_{H,H}$.
		\end{itemize}
		
		Therefore the obtained relations are the same as the relations in the previous lemma. The matrix entries $y_{I,J}=0$ with $d(I,J)\leq m-2$ do not contribute to the dimension of the Lie algebra. The entries $y_{I,J}$ with $d(I,J)= m-1$ give the contribution equal to the number of roots of $\Phi$, namely, $(n^2-n)$. Finally, the latter item allows us to express all entries $y_{I,I}$ as linear combinations of the entries $y_{K_j,K_j}, 1\leq j\leq n$, where each fundamental root of $\Phi$ occurs among the pairwise differences of the weights $K_j$. For instance, one
		can use the weights $\{1,\dots,m-1,p\}$, $m\leq p\leq n$, and $\{1,\dots,\hat{i},\dots,m+1\}$, $1\leq i< m$, see~\cite{VavPere}. Fig.~\ref{fig:DiagWeightsL2E5} shows their location in the weight diagram $(A_5,\varpi_2)$. Therefore the dimension of the Lie algebra $\Lie(\overline{G}_f(K))$ does not exceed $n^2-n+n=n^2$. The same argument is also applicable for the case of $\Lie(G_f(K))$. It suffices to set $\lambda(g)=1$. Again, we conclude that the dimension of $\Lie(G_f(K))$ does not exceed $n^2$.
		
		To conclude the proof of the theorem, we must reduce the dimension of $\Lie(G_f(K))$. For the sake of brevity, we conceive indices $I\in\bw{m}[n]$ as roots of the corresponding representation, and we write roots $\alpha=c_1\alpha_1+\dots+c_{n-1}\alpha_{n-1}\in A_{n-1}$ in the Dynkin form $c_1\dots c_{n-1}$, where $\alpha_j$ are the simple roots of $A_{n-1}$. For example, $\delta=1\dots 1$ is the maximal root of $A_{n-1}$. Suppose $K_1$ is the highest weight of the representation, and $I_2,\dots,I_k$ is the standard partition of the set $[n]\smallsetminus K_1$ into $m$-element subsets, i.e. $I_2>I_3>\dots >I_k$. Substituting $x^1:=e_{K_1},x^2:=e_{I_2},\dots,x^k:=e_{K_k}$, we get
		$$y_{K_1,K_1}+y_{I_2,I_2}+\dots+y_{I_k,I_k}=0.$$
		Further, note that for every $j$: $K_1-I_j=c^j_1\alpha_1+\dots+c^j_{n-1}\alpha_{n-1}$. Using already proven relations $y_{I,I}-y_{M,M}=y_{H,H}-y_{J,J}$ for $I-M=H-J$, express all diagonal entries $y_{I_j,I_j}$ as linear combinations of the entries $y_{K_j,K_j}$. Thus we find a non-trivial relation among $y_{K_j,K_j}$. Below we do this for arbitrary exterior power in detail.
		
		In this notation, $K_1-I_2=12\dots m\dots 210\dots 0$, $K_1-I_3=12\dots \underbrace{m\dots m}_{m+1\textrm{ times}}\dots 210\dots 0$, and in general $K_1-I_j=12\dots \underbrace{m\dots m}_{(j-2)\cdot m+1}\dots 21\underbrace{0\dots 0}_{n-mj}$ for  $2\leq j\leq k$. Recall that our numbering of the roots $K_j$ is such that $\alpha_m=K_1-K_2,\alpha_{m+1}=K_2-K_3,\dots,\alpha_{n-1}=K_{n-m}-K_{n-m+1},\alpha_{m-1}=K_2-K_{n-m+2},\alpha_{m-2}=K_{n-m+2}-K_{n-m+3},\dots,\alpha_1=K_{n-1}-K_{n}$ (for the exterior squares $\alpha_{m-1}=\alpha_1=K_2-K_{n-m+2}$). Then for $3\leq j\leq k$, we have
		\begin{align*}
			y_{K_1,K_1}-y_{I_j,I_j}&=(y_{K_{n-1},K_{n-1}}-y_{K_{n},K_{n}})+2(y_{K_{n-2},K_{n-2}}-y_{K_{n-1},K_{n-1}})+\dots
			\\
			&\quad+(m-1)(y_{K_{2},K_{2}}-y_{K_{n-m+2},K_{n-m+2}})
			\\
			&\quad+m\bigl((y_{K_{1},K_{1}}-y_{K_{2},K_{2}})+\dots+(y_{K_{m(j-2)+1},K_{m(j-2)+1}}-y_{K_{m(j-2)+2},K_{m(j-2)+2}})\bigr)
			\\
			&\quad+(m-1)(y_{K_{m(j-2)+2},K_{m(j-2)+2}}-y_{K_{m(j-2)+3},K_{m(j-2)+3}})+\dots
			\\
			&\quad+2(y_{K_{m(j-1)-1},K_{m(j-1)-1}}-y_{K_{m(j-1)},K_{m(j-1)}})
            \\
            &\quad+(y_{K_{m(j-1)},K_{m(j-1)}}-y_{K_{m(j-1)+1},K_{m(j-1)+1}})
			\\
			&=my_{K_1,K_1}+(m-1)y_{K_2,K_2}-y_{K_{m(j-2)+2},K_{m(j-2)+2}}-\dots
			\\
			&\quad-y_{K_{m(j-1)+1},K_{m(j-1)+1}}-y_{K_{n-m+2},K_{n-m+2}}-\dots-y_{K_n,K_n},
		\end{align*}
		and for $j=2$, we have
		\begin{align*}
			y_{K_1,K_1}-y_{I_2,I_2}&=(y_{K_{n-1},K_{n-1}}-y_{K_{n},K_{n}})+2(y_{K_{n-2},K_{n-2}}-y_{K_{n-1},K_{n-1}})+\dots
			\\
			&\quad+(m-1)(y_{K_{2},K_{2}}-y_{K_{n-m+2},K_{n-m+2}})
			\\
			&\quad+m(y_{K_{1},K_{1}}-y_{K_{2},K_{2}})
			\\
			&\quad+(m-1)(y_{K_{2},K_{2}}-y_{K_{3},K_{3}})+\dots
			\\
			&\quad+2(y_{K_{m-1},K_{m-1}}-y_{K_{m},K_{m}})+(y_{K_{m},K_{m}}-y_{K_{m+1},K_{m+1}})
			\\
			&=my_{K_1,K_1}+(m-2)y_{K_2,K_2}-y_{K_3,K_3}-\dots-y_{K_{m+1},K_{m+1}}
			\\
			&\quad-y_{K_{n-m+2},K_{n-m+2}}-\dots-y_{K_n,K_n}.
		\end{align*}
		It remains to add up all the obtained equalities with the equation $y_{K_1,K_1}+y_{I_2,I_2}+\dots+y_{I_k,I_k}=0$. Thus the final equation on diagonal entries is the following: 
		\begin{multline*}
			(m(k-1)-k)y_{K_1,K_1}+\bigl((m-1)(k-1)-1\bigr)y_{K_2,K_2}-y_{K_3,K_3}-\dots-y_{K_{n-m+1},K_{n-m+1}}
			\\
			-(k-1)y_{K_{n-m+2},K_{n-m+2}}-\dots-(k-1)y_{K_n,K_n}=0.
		\end{multline*}
		This is precisely the desired non-trivial linear relation among the entries $y_{K_j,K_j}$, which, over a field of
		any characteristic, shows that the dimension of our Lie algebra is 1 smaller than the above bound. Thus $\opn{dim}\Lie(G_f(K))\leq n^2-1$, as claimed.
	\end{proof}
	
	\begin{figure}[t]
		\centering
			\includegraphics[scale=0.6]{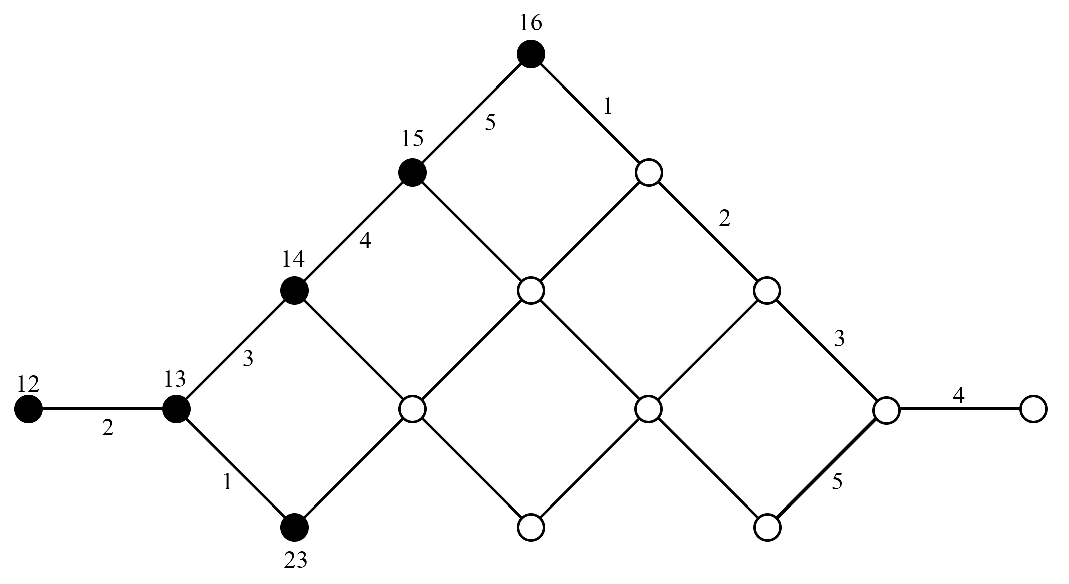}
			\caption{Diagonal weights in $(A_5,\varpi_2)$}
			\label{fig:DiagWeightsL2E5}
	\end{figure}
	
	Let us give an example of the proof calculations for the case of  $\bw{2}\E_6(R)$. Fig.~\ref{fig:DiagWeightsL2E5} shows the location of $K_j$ in the weight diagram. We have $y_{12,12}+y_{34,34}+y_{56,56}=0$ as the form is preserved.
	\begin{itemize}
		\item Since $12-34=\alpha_1+2\alpha_2+\alpha_3$, it follow that
		\begin{align*}
			y_{12,12}-y_{34,34}&=(y_{13,13}-y_{23,23})+2(y_{12,12}-y_{13,13})+(y_{13,13}-y_{14,14})
			\\
			&=2y_{12,12}-y_{14,14}-y_{23,23}.
		\end{align*}
		\item Since $12-56=\alpha_1+2(\alpha_2+\alpha_3+\alpha_4)+\alpha_5$, we have
		\begin{align*}
			y_{12,12}-y_{56,56}&=(y_{13,13}-y_{23,23})+2\bigl((y_{12,12}-y_{13,13})
			\\
			&\quad+(y_{13,13}-y_{14,14})+(y_{14,14}-y_{15,15})\bigr)+(y_{15,15}-y_{16,16})
			\\
			&=2y_{12,12}+y_{13,13}-y_{15,15}-y_{16,16}-y_{23,23}.
		\end{align*}
	\end{itemize}
	Adding up these three equations, we get a non-trivial linear relation among the entries $y_{K_j,K_j}$:
	$$y_{12,12}+y_{13,13}-y_{14,14}-y_{15,15}-y_{16,16}-2y_{23,23}=0.$$
	
	Now we verified all the conditions from Lemma~\ref{lem:Waterhouse} and are ready to complete the proof of Theorem~\ref{thm:StabAndGenerealPower}.
	
	\exGLandForms*
	\begin{proof}
		Consider the Cauchy--Binet morphism $\bw{m}$ of algebraic groups:
		$$\bw{m}\colon\GL_n\longrightarrow \GL_N.$$
		From Lemma~\ref{lem:VPKer}, it follows that the kernel of this morphism equals $\mu_m$. Proposition~\ref{prop:FormsInvariantUnderLE} implies that its image is contained in $\overline{G}_f$. Hence $\bw{m}$ induces a monomorphism of algebraic groups:
		$$\iota\colon\GL_n/\mu_m\longrightarrow \overline{G}_f.$$
		
		We wish to apply Lemma~\ref{lem:Waterhouse} to this morphism $\iota$. We know that $\dim(\bw{m}\GL_{n,K})=n^2$ (as an image of $\GL_{n,K}$ under Cauchet--Binet homomorphism with a finite kernel) for an algebraically closed field $K$. Theorem~\ref{thm:dimLie} implies that $\dim(\Lie(\overline{G}_{f,K}))\leq n^2$ with the same assumption on the field $K$. Therefore Condition~(1) of Lemma~\ref{lem:Waterhouse} holds true. As we discussed after Lemma \ref{lem:Waterhouse}, Conditions~(2) and~(3) are also satisfied.
		
		This means that we can apply Lemma~\ref{lem:Waterhouse} to conclude that $\iota$ is an isomorphism of affine group schemes over $\mathbb{Z}$. 
		
		The proof for the schemes $\bw{m}\SL_n$ and $G_f$ is similar and so it is omitted.
	\end{proof}
	
	\section{Difference between two exterior powers}\label{sec:DiffExterPwrs}
	
	The isomorphism $\iota\colon \bw{m}\GL_{n} \longrightarrow \overline{G}_{f}$ from the previous section shows that for arbitrary rings the class of transvections from $\bw{m}\GL_n(R)$ is strictly larger than the images $\bw{m}g$, $g\in\GL_n(R)$:
	\begin{align*}
		\bw{m}\bigl(\GL_n(R)\bigr) < \bw{m}\GL_n(R) \text{ for a general ring }R.
	\end{align*}
	
	Indeed, suppose $n\neq 2m$ (otherwise, one has to consider the argument for the corresponding connected component of the group). Then the exact sequence of affine group schemes
	$$1\longrightarrow\mu_m\longrightarrow\GL_n\longrightarrow\GL_n/\mu_m\longrightarrow 1$$
	gives an exact sequence of Galois cohomology
    \[
    1\longrightarrow\mu_m(R)\longrightarrow\GL_n(R)\longrightarrow\GL_n/\mu_m(R)
    \longrightarrow H^1(R,\mu_m)\longrightarrow H^1(R,\GL_n)\longrightarrow H^1(R,\GL_n/\mu_m).        
    \]
	
	The values of all these cohomology sets are well known, see~\cite[Chapter~III, \S2]{Knus},~\cite[\S9]{VavPere}, or in the case of exterior square~\cite{WaterhouseDet}. $H^1(R,\GL_n)$ classifies projective $R$-modules $P$ of rank $n$. In particular, $H^1(R,\GL_1)$ classifies invertible $R$-modules, i.e. finitely generated projective $R$-modules of rank $1$. The set $H^1(R,\GL_1)$ has a group structure induced by a tensor product. This group is called the Picard group $\opn{Pic}(R)$ of the ring $R$. Its elements are twisted forms of the free $R$-module $R$. 
	
	Let us consider the following exact sequence for description of $H^1(R,\mu_m)$:
	$$1\longrightarrow\mu_m\longrightarrow\GL_1\xrightarrow{(\blank)^m}\GL_1\longrightarrow 1,$$
	where $(\blank)^m$ is the $m^{th}$ power. Since $(\GL_1)^m(R)=R^{*m}$, we have
	$$1\longrightarrow R^{*}/R^{*m}\longrightarrow H^1(R,\mu_m)\longrightarrow\opn{Pic}(R)\longrightarrow\opn{Pic}(R),$$
	where the rightmost arrow is induced by $(\blank)^m$. Thus the cohomology group $H^1(R,\mu_m)$ classifies projective $R$-modules $P$ of rank $1$ together with the isomorphism $P^{\otimes m}=R$. 
	
	To describe the group $\GL_n/\mu_m(R)$ it remains to calculate the kernel of $H^1(R,\mu_m)\\\longrightarrow H^1(R,\GL_n)$. Observe that the morphism $\mu_m\longrightarrow\GL_n$ passes through $GL_1=\mathbb{G}_m$:
	\[
	     \xymatrix @=1pc{
			\mu_m\ar[rr]\ar[ddr]&&\GL_n\\
			\\
			&\GL_1\ar@{^(->}[uur]_{\opn{scalar}}}
	\]
	Since $H^1(R,\GL_n)$ classifies projective $R$-modules of rank $n$ and the embedding $\GL_1\hookrightarrow\GL_n$ sends $\lambda$ to $\lambda e$, the map $H^1(R,\GL_1)\longrightarrow H^1(R,\GL_n)$ sends an invertible module $P$ to $\bigoplus_{1}^n P$. Therefore the kernel of $H^1(R,\mu_m)\longrightarrow H^1(R,\GL_n)$ contains the whole group $R^*/R^{*m}$ and, in addition, elements $P$ of the Picard group $\opn{Pic}(R)$ such that $P^{\otimes m} \cong R$ and $\bigoplus_{1}^n P$ is free ($\cong R^n$).
	
	Summarizing both arguments, we see that \textit{the quotient  of $\bw{m}\GL_n(R)$ by $\bw{m}\bigl(\GL_n(R)\bigr)$ contains a copy of the group $R^*/R^{*m}$. The quotient by this group is isomorphic to a subgroup of the Picard group $\opn{Pic}(R)$ consisting of invertible modules $P$ over $R$ such that $P^{\otimes m} \cong R$ and $\bigoplus_{1}^n P$ is free}.
	
	\hspace{3mm}
	
	For the special linear group the argument is similar. The exact sequence of affine group schemes
	$$1\longrightarrow\mu_d\longrightarrow\SL_n\longrightarrow\SL_n/\mu_d\longrightarrow 1$$
	gives the exact sequence of Galois cohomology
    \[
    1\longrightarrow\mu_d(R)\longrightarrow\SL_n(R)\longrightarrow\SL_n/\mu_d(R)
    \longrightarrow H^1(R,\mu_d)\longrightarrow H^1(R,\SL_n)\longrightarrow H^1(R,\SL_n/\mu_d),
    \]
	where $d=\gcd(n,m)$. The values of all these cohomology sets are also well known, for instance see~\cite[Chapter~III, \S2]{Knus}.
	
	The determinant map $\opn{det}\colon\GL_n\longrightarrow\GL_1$ induces a map of pointed sets $(\opn{det})^1_*\colon H^1(R,\GL_n)\longrightarrow \opn{Pic}(R)$. Suppose $[T]\in H^1(R,\GL_n)$ is a class represented by a projective module $T$ of rank $n$. For any automorphism $\alpha$ of $T$, the determinant $\opn{det}(\alpha)\in R$ is the induced automorphism of the $n$-th exterior power $\bw{n}T$. Thus $(\opn{det})^1_*([T])=[\bw{n}T]$.
	
	Consider another exact sequence of groups:
	$$1\longrightarrow\SL_n(R)\longrightarrow\GL_n(R)\xrightarrow{\opn{det}} \GL_1(R)\longrightarrow 1.$$
	We describe the cohomology set $H^1(R,\SL_n)$. Let $M$ be a projective $R$-module of rank $n$ such that $\bw{n}M\cong R$. And let $\delta_M\colon \bw{n}M\longrightarrow R$ be a fixed isomorphism. An isomorphism $\psi\colon M\longrightarrow N$ is called an isomorhism of pairs $(M,\delta_M)\cong (N,\delta_N)$ if $\delta_N \circ \bw{n}\psi=\delta_M$. By $[M,\delta_M]$ denote the class of isomorphisms $(M,\delta_M)$. Then for any automorphism $\psi$ of $(M,\delta_M)$, we have $\delta_M \circ \bw{n}\psi=\delta_M$. This yields that $\opn{det}(\psi)=1$. Therefore the set $H^1(R,\SL_n)$ is determined by the classes $[M,\delta_M]$, i.e. by projective modules $M$ of rank $n$ together with the fixed isomorphism $\bw{n}M \cong R$. And the map $H^1(R,\SL_n)\longrightarrow H^1(R,\GL_n)$ corresponds to $[M,\delta_M]\mapsto [M]$.
	
	As before, we use the description of $H^1(R, \mu_d)$ in terms of $R^*/R^{*d}$ and projective modules $P$ of rank 1 such that $P^{\otimes d} \cong R$. The map $H^1(R,\mu_d)\longrightarrow H^1(R,\SL_n)$ sends a module $P$ to the pair $\bigr[\bigoplus_1^n P, \delta_P^{can}\bigr]$ where $\delta_P^{can}$ is the canonical isomorphism induced by the multiplication $\delta^{can}\colon R^{n} \longrightarrow R$.
	
	Summing up, we see that \textit{the quotient  of $\bw{m}\SL_n(R)$ by $\bw{m}\bigl(\SL_n(R)\bigr)$ contains a copy of the group $R^*/R^{* d}$. The quotient by this group consists of pairs $(P, \alpha)$ where $P$ is an element of the Picard group $\opn{Pic}(R)$ such that $P^{\otimes m} \cong R$ and $\alpha\colon \bigoplus_{1}^n P \longrightarrow R^n$ is an isomorphism such that $\delta^{can}_{P} = \delta^{can} \circ \alpha$}.
	
	\section{Exterior powers as the stabilizer of invariant forms II}\label{sec:StabII}
	
	In the previous sections, we completely analyzed the case of one invariant form. However if $\sfrac{n}{m} \not\in\mathbb{N}$, as we assume for this section, then the group $\bw{m}\GL_n(R)$ has only an \textit{ideal} of invariant forms. 
	
	Let us extend the definition of $q(x)$ from~\S\ref{sec:StabI}. Previously considered form $q(x) = q_{[n]}^{m}(x)$ is associated to the set $[n]=\{1,\dots,n\}$. In this section, we use forms associated to an arbitrary subsets of $[n]$ with fixed cardinality. Namely, we define $q^m_{V}(x)$ for an arbitrary $n_1$-subset $V\subseteq [n]$, where $\sfrac{n_1}{m}\in\mathbb{N}$:
	\begin{itemize}
		\item $q^m_V(x)=\sum \sign(I_{1}, \ldots, I_{\frac{n_1}{m}})\; x_{I_{1}}\ldots x_{I_{\frac{n_1}{m}}}$ for even $m$;
		\item $q^m_V(x)=\sum \sign(I_{1}, \ldots, I_{\frac{n_1}{m}})\; x_{I_{1}}\wedge \ldots \wedge x_{I_{\frac{n_1}{m}}}$ for odd $m$,
	\end{itemize}
	where the sums in the both cases range over all unordered partitions of the set $V$ into $m$-element subsets $I_{1}, \ldots, I_{\frac{n_1}{m}}$.
	
	As usual, $f^m_{V}(x^1,\dots,x^k)$ denotes the \textit{$[$ full $]$ polarization} of $q^m_V(x)$, where $k:=\frac{n_1}{m}$. We ignore the power $m$ in the notation $f^m_{V}(x^1,\dots,x^k)$ and $q^m_V(x)$ if it is clear from context.
	
	Let $n = lm +r$ where $l,r \in \mathbb{N}$ and $l$ is the maximal such. Consider the ideal $F = F_{n,m}$ of the ring $\mathbb{Z}[x_I]$ generated by the forms $f_{V}(x^1,\dots,x^k)$ for all possible $ml$-element subsets $V \subsetneq [n]$.  We define the extended Chevalley group $\overline{G}_F(R)$ as the group of linear transformations preserving the ideal $F$:
	\begin{multline*}
		\overline{G}_F(R):=\{g\in\GL_N(R)\mid\text{ there exist } \lambda_{V_1},\dots,\lambda_{V_p}\in R^*, c(V_k,V_l)\in R \text{ such that }
		\\
		f_{V_j}(gx^1,\dots,gx^k)=\lambda_{V_j}(g)f_{V_j}(x^1,\dots,x^k)+\sum\limits_{l\neq j}c({V_j},{V_l})\cdot f_{V_l}(x^1,\dots,x^k)
        \\
        \text{ for all }j \text{ satisfying } 1\leq j\leq p\}.
	\end{multline*}
	
	First we must show that $\overline{G}_F$ is a group scheme. We use the following standard argument. 
	
	Let $f_1,\dots,f_s$ be arbitrary polynomials in $t$ variables with coefficients in a commutative ring $R$. We are interested in the linear changes of variables $g\in\GL_t(R)$ that preserve the condition that all these polynomials simultaneously vanish. In other
	words, we consider all $g\in\GL_t(R)$ preserving the ideal $A$ of the ring $R[x_1,\dots,x_t]$ generated by $f_1,\dots,f_s$.  It is well known (see, e.g.~\cite[Lemma 1]{DixonAffSch} or \cite[Proposition~1.4.1]{WaterhouseDet}) that the set $G_A(R)=\Fix_R(A)=\Fix_R(f_1,\dots,f_s)$ of all such linear variable changes $g$ forms a group. For any $R$-algebra $S$ with $1$, we can consider $f_1,\dots,f_s$ as polynomials with coefficients in $S$. Thus the group $G(S)$ is defined for all $R$-algebras. It is clear that $G(S)$ depends functorially on $S$. It is easy to provide examples showing that $S\mapsto G(S)$ may fail to be an affine group scheme over $R$. This is due to the fact that $G_A(R)$ is defined by congruences, rather than equations, in its matrix entries. However in Theorem~1.4.3 of~\cite{WaterhouseDet} a simple sufficient condition was found, that guarantees that $S\mapsto G(S)$ is an affine group scheme. Denote by $R[x_1,\dots,x_t]_r$ the submodule of polynomials of degree at most $r$. The following lemma is Corollary~1.4.6 in~\cite{WaterhouseDet}.
	\begin{lemma}\label{lem:StabAffGrpSch}
		Let $f_{1}, \dots, f_{s} \in \mathbb{Z}[x_{1}, \dots , x_{t}]$ be polynomials of degree at most $r$ and let $A$ be the ideal they generate. Then for the functor $S \mapsto \Fix_{S}(f_{1}, \dots, f_{s})$ to be an affine group scheme, it suffices that the rank of the intersection $A\cap R[x_{1}, \dots, x_{t}]_{r}$ does not change under reduction modulo any prime $p \in \mathbb{Z}$. This is true in particular if all generators of $A$ remain independent modulo $p$ for all prime $p$.
	\end{lemma}
	
	We apply this lemma to the case of the ideal $F$ in $\mathbb{Z}[x_I]$.
	\begin{lemma}
		Let $n=ml+r$, where $m,l\in\mathbb{N}$. Then the functor $R\mapsto\overline{G}_F(R)$ is an affine group scheme over $\mathbb{Z}$.
	\end{lemma}
	\begin{proof}
		Let us show that for any prime $p$ the polynomials $f_{V_j}$ are linear independent modulo $p$. Indeed, specializing $x_I$ appropriately, we can guarantee that one of these polynomials takes value $\pm 1$, while all other vanish. Let $I_{1}\sqcup \dots\sqcup I_{l} = V_j$ be a partition of some $ml$-element subset $V_j\subset[n]$. Set $x_{I_{j}}:=1$ for $i=1, \dots, l$ and $x_I :=0$ otherwise. The monomial $x_{I_1}\dots x_{I_l}$ occurs only in one form corresponding to the partition $V_j=I_{1}\sqcup \dots\sqcup I_{l}$. Thus the value of the polynomial $f_{V_j}$ is $\sign(I_{1}, \ldots, I_l)=\pm 1$.
	\end{proof}
	
	Our immediate goal is to prove the coincidence of $\overline{G}_F$ and $\bw{m}\GL_n$. Lemma~\ref{lem:Waterhouse} is useful for this again. Using the results of the previous two sections, we only must verify coincidence of $\bw{m}\GL_n(K)$ and $\overline{G}^0_F(K)$ for algebraically closed fields and smoothness of $\overline{G}_F$. 
	
	The proof of the following proposition is completely analogous to the proof of Proposition~\ref{prop:LGLeqStabgoodOverK}.
	\begin{prop}\label{prop:LGLeqStabIdealOverK}
		Suppose $K$ is an algebraically closed field. Then
		$$\bw{m}\GL_n(K)=\overline{G}^0_F(K).$$
	\end{prop}
	
	To verify that the scheme $\overline{G}_F$ is smooth one needs to evaluate the dimension of the Lie algebra. As above, it is possible to identify the Lie algebra $\Lie(\overline{G}_F(K))$ with a homomorphism kernel sending $\delta$ to $0$ in $K[\delta]$. Thus $\Lie(\overline{G}_F(K))$ consists of the matrices $g=e+y\delta$ where $y\in\M_N(K)$ satisfying the following conditions
	$f_{V_j}(gx^1,\ldots,gx^k)=\lambda_{V_j}(g)f_{V_j}(x^1,\ldots,x^k)+\sum\limits_{l\neq j}c(V_j,V_l)f_{V_l}(x^1,\ldots,x^k)$ for $1\leq j\leq p$ and $x^1,\ldots,x^k\in K^N$.
	\begin{theorem}\label{dimLieIdeal}
		For any field $K$ the dimension of the Lie algebra $\Lie(\overline{G}_F(K))$ does not exceed $n^2$.
	\end{theorem}
	
	\begin{proof}
		Let $g$ be a matrix satisfying the above conditions for all $1\leq j\leq p$ and $x^1,\dots,x^k\in K^N$. Plugging in $g=e+y\delta$ and using that the form $f_{V_j}$ is $k$-linear, we get
		\begin{multline*}
			\delta\bigl(f_{V_j}(yx^1,x^2,\ldots,x^k)+\ldots+f_{V_j}(x^1,\ldots,x^{k-1},yx^k)\bigr)\\
			=(\lambda_{V_j}(g)-1)f_{V_j}(x^1,\ldots,x^k)+\sum\limits_{l\neq j}c(V_j,V_l)f_{V_l}(x^1,\ldots,x^k)
		\end{multline*}
		for all $1\leq j\leq p$.
		
		Now we show that the entries of the matrix $y$ are subject to the same linear dependences, as in Theorem~\ref{thm:dimLie}. By the very definition of the forms, $f_{V_j}(e_{I_1},\ldots,e_{I_k})=0$ except the cases when $\{I_l\}$ is a partition of the set $V_j=I_1\sqcup\ldots\sqcup I_k$.
		\begin{itemize}
			\item If $d(I,J)\leq m-2$ ($|I\cup J|\geq m+2$), then $y_{I,J}=0$. Indeed, then there is a set of pairwise disjoint indices $I_2,\ldots,I_k\in\bw{m}\bigl(V_j\smallsetminus I\bigr)$ such that $d(J,I_2)\geq 1,\, d(J,I_3)\geq 1$ and $d(J,I_4)=\dots=d(J,I_k)=0$. Set $x^1:=e_{J},x^l:=e_{I_l}, 2\leq l\leq k$. Then $f_{V_j}(x^1,yx^2,\ldots,x^k)=\ldots=f_{V_j}(x^1,x^2,\ldots,yx^k)=0$. It follows that $f_{V_j}(yx^1,x^2,\ldots,x^k)=\pm y_{I,J}=0$.
			\item If $d(I,J)=d(M,H)=m-1$, then $y_{I,J}=\pm y_{H,M}$. Here there is a set of pairwise disjoint indices $M,I_3,\ldots,I_k \in \bw{m}\bigl(V_j\smallsetminus I\bigr)$ such that $d(J,M)=1$ and $d(J,I_3)=\dots=d(J,I_k)=0$. Set $x^1:=e_{J},x^2:=e_{M},x^l:=e_{I_l}, 3\leq l\leq k$ and denote by $H$ the index $V_j\smallsetminus(J\cup I_2\cup\ldots\cup I_k)$. Then $f_{V_j}(x^1,x^2,yx^3,\ldots,x^k)=\ldots=f_{V_j}(x^1,x^2,\ldots,yx^k)=0$. It follows that $f_{V_j}(yx^1,x^2,\ldots,x^k)+f_{V_j}(x^1,yx^2,x^3\ldots,x^k)=0$. But $f_{V_j}(yx^1,x^2,\ldots,x^k)=\sign(I,M,I_3,\ldots,I_k)\cdot y_{I,J}$, and $f_{V_j}(x^1,yx^2,x^3,\ldots,x^k)=\sign(J,H,I_3,\ldots,I_k)\cdot y_{H,M}$.
			\item Finally, for diagonal entries the following condition holds $y_{I,I}-y_{M,M}=y_{H,H}-y_{J,J}$, where $d(I,J)=d(H,M)=0$ and $I\cup J=H\cup M$. In this case there is a set of pairwise disjoint indices $I_3,\ldots,I_k\in\bw{m}\bigl(V_j\setminus(I\cup J)\bigr)$. In other words, $I,J,I_3,\dots,I_k$ is a partition of the set $V_j$. Put $x^1:=e_I,x^2:=e_J,x^l:=e_{I_l}$ where $3\leq l\leq k$. Since $f_{V_l}(x^1,\ldots,x^k)=0$ for all $l\neq j$, we get
			$$(\lambda_{B_j}(g)-1)=\delta(y_{I,I}+y_{J,J}+y_{I_3,I_3}+\ldots+y_{I_k,I_k}).$$
			On the other hand, $H,M,I_3,\dots,I_k$ is partition of the set $V_j$ too, where $I\cup J=H\cup M$. Substituting $x^1:=e_H,x^2:=e_M,x^l:=e_{I_l}$ for all $3\leq l\leq k$, we have
			$$(\lambda_{B_j}(g)-1)=\delta(y_{M,M}+y_{H,H}+y_{I_3,I_3}+\ldots+y_{I_k,I_k}).$$
			Combining the obtained qualities, we see that $y_{I,I}+y_{J,J}=y_{M,M}+y_{H,H}$.
		\end{itemize}
		Thus, as in the proof of Theorem~\ref{thm:dimLie}, it turns out that the dimension of the Lie algebra $\Lie(\overline{G}_F(K))$ does not exceed $n^2$: the entries $y_{I,J}$ do not contribute to the dimension when $d(I,J)\leq m-2$, they make a contribution $n(n-1)$ when $d(I,J)=m-1$ and, finally, they make a contribution $n$ for $d(I,J)=m$.
	\end{proof}
	
	Consequently we verified all the condition from Lemma~\ref{lem:Waterhouse} and can conclude that $\bw{m}\GL_n$ equals the stabilizer of $F$. The proof is similar to the proof of Theorem~\ref{thm:exGLandForms}.
	\exGLandFormsIdeal*
	
	\section{Normalizer Theorem}\label{sec:transgood}

    We modify our approach in proving Theorem~\ref{thm:NNTran} by contrasting it with Theorems~\ref{thm:exGLandForms} and~\ref{thm:exGLandFormsIdeal}. Specifically, in Theorem~\ref{thm:NNTran2}, we establish that the functors of $R$-points coincide for the group schemes under consideration, for arbitrary ring $R$.
    
    \begin{theorem}\label{thm:NNTran2}
    If $n \geq 4$ and $\sfrac{n}{m}$ is an integer greater than 2, then
        for any commutative ring $R$, we have
        \[
            N\bigl(\bw{m}\E_n\bigr)(R) = N\bigl(\bw{m}\SL_{n}\bigr)(R) = \Tran\bigl(\bw{m}\E_n, \bw{m}\SL_{n}\bigr)(R)
            = \Tran\bigl(\bw{m}\E_n, \bw{m}\GL_{n}\bigr)(R) = \bw{m}\GL_{n}(R),
        \]
		where all normalizers and transporters are taken inside the group scheme $\GL_{\binom{n}{m}}$.
    \end{theorem}

Before proving the theorem, we address the issue of group-theoretic vs. scheme-theoretic objects appearing in the theorem. Classically, the theorem is formulated with normalizers and transporters as abstract groups. For example, the (group of $R$-points of the) transporter
\[
    \Tran(\bw{m}\E_n, \bw{m}\SL_n)(R) :=  \Big\{g\in \GL_{\binom{n}{m}}(R)\mid z^g\in \bw{m}\SL_n(\tilde{R})
    \text{ for all $R$-algebras }\tilde{R}\text{ and }z\in \bw{m}\E_n(\tilde{R})\Big\}
\]
should be replaced with the transporter (as an abstract group) 
\[
    \Tran(\bw{m}\E_n(R), \bw{m}\SL_n(R)) :=  \Big\{g\in \GL_{\binom{n}{m}}(R)\mid z^g\in \bw{m}\SL_n(R)
    \text{ for all }z\in \bw{m}\E_n(R)\Big\}.
\]
In this presentation, we immediately see the inclusion 
$$\Tran(\bw{m}\E_n, \bw{m}\SL_n)(R) \leq \Tran(\bw{m}\E_n(R), \bw{m}\SL_n(R)).$$
The next proposition~\cite[Lemma 4.1, Proposition 4.3]{LubStepSub} presents other more nontrivial inclusions between different version of the normalizers and transporters.

    \begin{prop}\label{prop:LubStepSub}
    In the assumptions of Theorem~\ref{thm:NNTran} and~\ref{thm:NNTran2}, the following inclusions hold:
    $$\begin{alignedat}{3}
        N\bigl(\bw{m}\E_n(R)\bigr) &= \Tran\bigl(\bw{m}\E_n(R), \bw{m}\SL_{n}(R)\bigr) &\geq N\bigl(\bw{m}\SL_{n}(R)\bigr),\\
        N\bigl(\bw{m}\E_n\bigr)(R) &= \Tran\bigl(\bw{m}\E_n, \bw{m}\SL_{n}\bigr)(R) &= N\bigl(\bw{m}\SL_{n}\bigr)(R).
    \end{alignedat}$$
    \end{prop}

The question of when all these groups coincide is a quite tricky. For example, \cite[Proposition 4.5]{LubStepSub} proves it in a general situation for algebras over infinite fields; and \cite{LubStepSubE} proves it for our case for an arbitrary $R$ with $\opn{char}(R) \neq 2$.

\begin{proof}[Proof of Theorem~\ref{thm:NNTran2} $($and Theorem~\ref{thm:NNTran}$)$]
First, the equality of the first three sets follows from Proposition~\ref{prop:LubStepSub}. Moreover, a standard Lie-theoretic argument~\cite[Chapter 4, Cor.3.9]{SanRitt2017} shows that $N(\bw{m}\SL_n)(R)$ is a group scheme, so all three of them are.

Second, we prove the inclusion $\bw{m}\GL_{n}(R) \leq N\bigl(\bw{m}\SL_n\bigr)(R)$ via Theorem~\ref{thm:StabAndGenerealPower}. Indeed, $g \in \bw{m}\GL_n(R)$ implies that $g$ stabilizes the form $f$ up to a scalar $\lambda(g)$. Then, for an arbitrary $R$-algebra $\tilde{R}$, the element $g b g^{-1}$ stabilizes $f$ as $\lambda(g)\lambda(g^{-1}) = 1$ and $b \in \bw{m}\SL_n(\tilde{R})$ stabilizes $f$.
		
Third, we show the inclusion $\Tran\bigl(\bw{m}\E_n, \bw{m}\SL_{n}\bigr)(R) \leq \bw{m}\GL_n(R)$. We pick an element $g \in \Tran\bigl(\bw{m}\E_n,\bw{m}\SL_{n}\bigr)(R)$ and an element $h\in \bw{m}\E_n(\tilde{R})$. Then $a:= ghg^{-1}$ belongs to $\bw{m}\SL_{n}(\tilde{R})$, and thus
$$f(ax^1,\ldots,ax^k)=f(x^1,\ldots,x^k).$$
Substituting $(gx^1,\dots,gx^k)$ for $(x^1,\dots,x^k)$, we get
$$f(ghx^1,\ldots,ghx^k)=f(gx^1,\ldots,gx^k).$$
Consider the form $D\colon R^N\times\dots\times R^N\longrightarrow R$ defined by the rule
$$D(x^1,\ldots,x^k):=f(gx^1,\ldots,gx^k).$$
By our assumption, one has
$$D(hx^1,\ldots,hx^k)=D(x^1,\ldots,x^k)$$
for all $h\in\bw{m}\E_n(\tilde{R})$. Hence the form $D$ is invariant under the action of $\bw{m}\E_n(\tilde{R})$. Thus Proposition~\ref{prop:FormsGoodRing} shows us
$$D(x^1,\ldots,x^k)=\lambda\cdot f(x^1,\ldots,x^k)\text{ for some }\lambda\in \tilde{R}.$$
As the transporter is a group, we can plug in $g^{-1}$ instead of $g$. Thereby we conclude that $\lambda$ is invertible. This shows that $g$ belongs to the group $\overline{G}_f(\tilde{R})$. But initially $g\in\GL_N(R)$, so $g$ belongs to $\overline{G}_f(R)$ which by Theorem~\ref{thm:exGLandForms} coincides with $\bw{m}\GL_n(R)$.

Finally, the equality $\Tran\bigl(\bw{m}\E_n, \bw{m}\SL_{n}\bigr)(R) = \Tran\bigl(\bw{m}\E_n, \bw{m}\GL_{n}\bigr)(R)$ follows from Proposition~\ref{prop:FormsGoodRing} and Theorem~\ref{thm:exGLandForms}. Indeed, if $z^g$ (with $z$ and $g$ are from $\tilde{R}$-points of the group schemes) belongs to $\bw{m}\GL_{n} \cong \overline{G}_f$, then the scalar of semi-invariancy is $\det(z^g) = \det(z) = 1$. Therefore $z^g$ belongs to $G_f \cong \bw{m}\SL_{n}$.
\end{proof}
\begin{remark}
We turn to the structure theory of Lie groups for proving that $N(\bw{m}\SL_n)$ is a group. Alternatively, we can employ the proved isomorphism $N(\bw{m}\SL_n) \cong G_{f}$ to deduct explicit equations, as in~\cite{LubNek18}, for the functor $N(\bw{m}\SL_n)$ and, using Jacobi's complementary formula, verify that they cut out a group scheme.
\end{remark}
\begin{remark}
The equivalence $\Tran\bigl(E(\Phi, -), G(\Phi, -)\bigr) \cong \Tran\bigl(E(\Phi, -), \overline{G}(\Phi, -))$ holds in a general situation. It is enough to use the argument of \cite[Lemma 4.1]{LubStepSub} and immediate generalization of the main theorem of \cite{HazVav} to the extended Chevalley group $\overline{G}(\Phi, -)$.
\end{remark}
		


	
\end{document}